\documentclass[10pt]{article}
\usepackage[active]{srcltx}
\usepackage{epsfig}
\usepackage{latexsym}
\usepackage{amsmath,amssymb,amsthm,amsfonts}
\usepackage{graphics}

\newcommand{\dst}{\displaystyle}
\newcommand{\refe}[1]{(\ref{#1})}

\newtheorem{theorem}{Theorem}[section]
\newtheorem{lemma}[theorem]{Lemma}

\theoremstyle{definition}

\newtheorem{corollary}[theorem]{Corollary}
\theoremstyle{remark}
\newtheorem{remark}[theorem]{Remark}

\numberwithin{equation}{section}

\newcommand{\abs}[1]{\left\vert#1\right\vert}

\newcommand{\ima}{\mathrm{i}}

\begin{document}

\title{On new rational approximants to $\zeta(3)$}
\author{J. Arves\'u\thanks{The research of J. Arves\'u was
partially supported by the research grant MTM2009-12740-C03-01 of the
Ministerio de Educaci\'on y Ciencia of Spain, travel grant from Fundaci\'on Caja Madrid and grant CC-G08-UC3M/ESP-4516 from Comunidad Aut\'o\-no\-ma de Madrid.} and A. Soria-Lorente
 \medskip \\
Department of Mathematics, Universidad Carlos III de Madrid,\\
Avda. de la Universidad, 30, 28911, Legan\'es, Madrid, SPAIN
}
%


\date{\today }


\maketitle
\begin{abstract}
New (infinitely many) rational approximants to $\zeta(3)$ proving its irrationality are given. The recurrence relations for the numerator and denominator of these approximants as well as their continued fraction expansions are obtained. A comparison of our approximants with Ap\'ery's approximants to $\zeta(3)$ is shown.
\end{abstract}



\section{Introduction}

The transcendence of the values of Riemann's zeta function
$\zeta(s)=\sum_{n=1}^{\infty}n^{-s}$, $\mbox{Re}\,s>1$, at even integers is a well-known fact. Indeed,
\begin{equation*}
\zeta(2k)=(-1)^{k-1}\frac{(2\pi)^{2k}B_{2k}}{2(2k)!},\quad k=1,2,\ldots,
\end{equation*}
where $B_{2k}\in\mathbb{Q}$ are the Bernoulli numbers. By $\mathbb{Q}$ we denote the set of rational numbers. However, few partial results about
the arithmetic properties of the numbers $\zeta(2k+1)$, $k=1,2,\dots$, have been obtained \cite{apery,ball-rivoal,ri1,ri2,zudilin-0,zudilin-1,zudilin-3}. Among these results we highlight below the pioneer contribution of R. Ap\'{e}ry \cite{apery}, who proved in 1978 the irrationality of $\zeta(3)$ (see \cite{van-der-poorten} for complementary information). Ap\'ery's proof is based on the existence of the recurrence relation  
\begin{equation}
(n+1)^{3}y_{n+1}-(2n+1)(17n^2+17n+5)y_{n}+n^{3}y_{n-1}=0,\quad n=1,2,\dots,
\label{apery-diff-eq-1}
\end{equation}
such that the two independent solutions, namely the sequence of integers $(q_n)_{n\geq0}$ defined by the initial conditions $q_0=1$, $q_1=5$, and 
the sequence of rationals $(p_n)_{n\geq0}$ determined by the initial conditions $p_0=0$, $p_1=6$, form good rational approximants $(p_n/q_n)_{n\geq0}$ 
to the number $\zeta(3)$. Indeed, this sequence can be expressed as the sequence of convergents of the irregular continued fraction
\begin{equation}
\frac{6\mid }{\mid 5}-\frac{1\mid }{\mid 117}-\frac{
64\mid }{\mid 535}-\cdots -\frac{n^{6}\mid }{\mid 34n^3+51n^2+27n+5}-\cdots  \label{apery_frac}
\end{equation}
Here, the convergence is so fast that reveals the irrationality of $\zeta(3)$, since  the inclusions $q_n,p_{n}l_{n}^{3}\in\mathbb{Z}$ (the set of integers numbers), where $l_{n}$ denotes the least common multiple of $\{1,2,\dots,n\}$, combined with the prime number theorem, $l_n=O\left(e^{(1+\epsilon)n}\right)$ for any $\epsilon>0$, and Poincar\'e's theorem \cite{perron,poincare} yield for the error-term sequence $r_n=q_{n}\zeta(3)-p_{n}$, the estimation 
\begin{equation*}
\limsup_{n\to\infty}\sqrt[n]{l_{n}^{3}|r_n|}\leq e^{3}(\sqrt{2}-1)^{4}<1.
\label{apery-estimation}
\end{equation*}
Moreover, for any constant $\mu>1+(\log(\sqrt{2}+1)^4+3)/(\log(\sqrt{2}+1)^4-3)\approx13.417820\ldots$
the inequality $\abs{\zeta(3)-p/q}<q^{-\mu}$ has only a finite number of solutions in $\mathbb{Z}$. Although, the best known 
irrationality measure for $\zeta(3)$ ($\mu\approx5.513891\ldots$) was given more than twenty years later \cite{rin-viola}.

After Ap\'ery's remarkable achievement some variants of his proof appeared. In particular, \cite{beuker-1} and \cite{sorokin-1} consider triple integrals 
for the estimation of error-term sequence $r_n$ of Diophantine approximations to $\zeta(3)$, and in \cite{nesterenko} were used the following complex countour integral 
\begin{equation*}
q_{n}\zeta \left( 3\right) -p_{n}=\frac{1}{2\pi \ima}\int_{1/2-\ima\infty
}^{1/2+\ima\infty }R_{1,n}(\nu -1) \left( \frac{\pi }{\sin \pi \nu }
\right) ^{2}d\nu ,
\end{equation*}
for the same purpose, where
\begin{equation}
R_{1,n}(\nu) =\frac{(-\nu) _{n}^{2}}{(\nu+1)
_{n+1}^{2}},\quad  n\in\mathbb{N}.\label{RationalFNest}
\end{equation}
By $\mathbb{N}$ we denote the set of all positive integers and $(\nu) _{n}=\nu(\nu+1) \cdots (\nu+n-1)$, $(\nu)_{0}=1$,
is the Pochhammer symbol.

In \cite{arvesu} Ap\'ery's proof and all the aforementioned variants \cite{beuker-1,nesterenko,sorokin-1} were put in the same context using a straightforward 
technique for calculation of Wronskian determinant for sequences $(p_n)_{n\geq0}$ and $(q_n)_{n\geq0}$ based on orthogonal forms, which allow to deduce 
equation \refe{apery-diff-eq-1}. Moreover, both type of integrals involved in the estimation of the error-term sequences are linked (see also \cite{nesterenko1}).

Despite the time that has passed from Ap\'ery's discovery only few sequences of rational approximants different from Ap\'ery's one, which prove the irrationality 
of $\zeta(3)$ are known \cite{hata,nesterenko-2,rin-viola}.

In this paper we give in Section \ref{section2} new (infinitely many) rational approximants to $\zeta(3)$ leading to its irrationality. 
More precisely, we obtain a set of rational approximants classified in twelve types $p_{n}^{(i,j)}/q_{n}^{(i,j)}$ $(1\leq i\leq 4,\,\,1\leq j\leq 3)$ 
 depending on certain parameters, which allow to generate infinitely many approximants where the error-term sequence $r_{n}^{(i,j)}=q_{n}^{(i,j)}\zeta \left( 3\right) -p_{n}^{(i,j)}$ is given by
\begin{equation*}
\dst r_{n}^{(i,j)}=
\frac{1}{2\pi \ima}\int_{1/2-\ima\infty }^{1/2+\ima\infty }R_{1,n}^{(i,j)}\left( \nu
-1\right) \left( \frac{\pi }{\sin \pi \nu }\right) ^{2}d\nu,
\end{equation*}
and functions $R_{1,n}^{(i,j)}(\nu)$ are modifications of \refe{RationalFNest}.
In Section \ref{rapp} we link each of the above rational approximants with a corresponding simultaneous rational approximation problem near infinity. 
As a result in Section \ref{aperyRR} we compute the Wronskian determinant for sequences 
$\left(p_n^{(i,j)}\right)_{n\geq0}$ and $\left(q_n^{(i,j)}\right)_{n\geq0}$ and from it we deduce in a straightforward manner the corresponding recurrence relations for these approximants as well as their continued fraction expansions. Lastly, several comparisons between the aforementioned rational approximants and Ap\'ery's approximants are shown.

\section{Rational approximants to $\protect\zeta \left( 3\right) $\label{section2}}

Define the following sequences of rational functions
\begin{equation*}
\left( R_{1,n}^{(i,j)}(t) \right)_{n\geq0},\quad i=1,2,3,4,\quad j=1,2,3,
\end{equation*}
where
\begin{equation}
R_{1,n}^{(i,j)}(t) =R_{1,n}(t) \theta^{(i,j)}(t),
\label{RationalFi,j}
\end{equation}
and
\begin{equation}
\theta^{(i,j)}(t)=\dfrac{[\delta_{1,j}(t+n+\rho)+\delta_{2,j}(t+\vartheta n+1)+\delta_{3,j}(\upsilon t-\chi n-\psi)]^{\delta _{i,2}+1}}{(t-n+1)^{2-\delta_{i,1}}(t+n+1)^{-\delta_{i,3}}},\label{Perturbs}
\end{equation}
with $\rho\in\mathbb{N}\setminus\{1\}$, $\vartheta \in \mathbb{Z}\backslash
\left\{-1,0,1\right\}$, $\upsilon,\chi\in \mathbb{N}$ $(\chi \geq \upsilon)$, $\psi\in\mathbb{N}\cup\{0\}$. By $\delta_{i,j}$ we denotes the Kronecker delta function. 

The partial fraction expansion of \refe{RationalFi,j} is given by
\begin{equation}
R_{1,n}^{(i,j)}(t) =\ \sum_{k=0}^{n}\frac{
a_{k,n}^{(i,j)}}{t+k+1}-\sum_{k=0}^{n-\delta _{i,3}}\frac{
b_{k,n-\delta _{i,3}}^{(i,j)}}{\left( t+k+1\right) ^{2}},
\label{R1_simple}
\end{equation}
with coefficients
\begin{eqnarray}
b_{k,n-\delta _{i,3}}^{(i,j)} &=&-\lim_{t\to-(k+1)}\left( t+k+1\right)
^{2}R_{1,n}^{(i,j)}(t),\quad
k=0,\ldots ,n-\delta _{i,3},  \notag \\
&=&-\binom{n+k}{k}^{2}\binom{n}{k}^{2}\theta ^{(i,j)}\left(
-k-1\right),  \label{bk,n}
\end{eqnarray}
and
\begin{eqnarray}
a_{k,n}^{(i,j)} &=&\mathop{\mathrm{Res}}_{t=-k-1}{R_{1,n}^{(i,j)}(t)},\quad
k=0,\ldots ,n-\delta _{i,3},  \notag \\
&=&2b_{k,n-\delta _{i,3}}^{(i,j)}\left[
H_{n+k-1}-2H_{k}+H_{n-k-\delta _{i,3}}-\varphi ^{(i,j)}\left(
-k-1\right) \right],  \label{ak,n}
\end{eqnarray}
where
\begin{equation*}
\varphi ^{(i,j)}(t) =\frac{d}{dt}\log \theta^{(i,j)}(t),
\end{equation*}
and $H^{(r)}_k$ denotes the Harmonic Number $k$ of order $r$ ($H^{(1)}_k=H_k$ and $H_0=0$). Moreover, the following particular situation yield
\begin{eqnarray}
a_{n,n}^{\left( 3,j\right) } &=&\lim_{t\to-(n+1)}\left( t+n+1\right) R_{1,n}^{\left(
3,j\right) }(t)\notag \\
&=&\frac{\left( -1\right) ^{\delta _{3,j}}}{n^{2-\delta _{2,j}}}\binom{2n-1}{
n-1}^{2}\sigma _{3,j},\quad j=1,2,3,  \label{an,n}
\end{eqnarray}
where $\sigma _{3,1}=\rho -1$, $\sigma _{3,2}=\vartheta -1$, $\sigma _{3,3}=\upsilon(n+1) +\chi n+\psi$.
Furthermore, from the above explicit expressions of coefficients \refe{bk,n}-\refe{an,n} the following inclusions 
\begin{equation}
n^{\omega _{i,j}}b_{k,n-\delta
_{i,3}}^{(i,j)},n^{\omega
_{i,j}}l_{n}a_{k,n}^{(i,j)}\in \mathbb{Z},\quad k=0,\ldots,n,
\label{inclu}
\end{equation}
hold.

Denoting $R_{2,n}^{\left( i,j\right)
}(t) =\frac{d}{dt}R_{1,n}^{\left(i,j\right)}\left(t\right)$ one gets
\begin{equation}
R_{2,n}^{(i,j)}(t) =2R_{1,n}^{\left( i,j\right)
}(t) \left[ \sum_{k=0}^{n-2}\frac{1}{t-k}-\sum_{k=1}^{n-\delta
_{i,3}+1}\frac{1}{t+k}+\varphi ^{(i,j)}(t) \right].
\label{R_{2,n}}
\end{equation}

For the main result of this section we will need the following Lemmas.
\begin{lemma}
The following relation is valid
\begin{equation}
\sum_{t=0}^{\infty}R_{2,n}^{(i,j)}\left(
t\right) =q_{n}^{(i,j)}\zeta \left( 3\right) -p_{n}^{\left(
i,j\right) },\quad n=0,1,\dots,\label{Remainderi,j}
\end{equation}
where
\begin{equation}
q_{n}^{(i,j)}=\ 2\sum_{k=0}^{n-\delta _{i,3}}b_{k,n-\delta
_{i,3}}^{(i,j)},\quad  p_{n}^{\left( i,j\right)
}=2\sum_{k=1}^{n-\delta _{i,3}}b_{k,n-\delta _{i,3}}^{\left( i,j\right)
}H_{k}^{\left( 3\right) }-\sum_{k=1}^{n}a_{k,n}^{\left( i,j\right)
}H_{k}^{\left( 2\right) }.  \label{RApproxs}
\end{equation}
Coefficients $a_{k,n}^{(i,j)}$ and $b_{k,n-\delta _{i,3}}^{\left(
i,j\right) }$ are given in \refe{bk,n}-\refe{an,n}. Moreover, 
$n^{\omega _{i,j}}q_{n}^{(i,j)}\in\mathbb{Z}$, and $n^{\omega_{i,j}}l_{n}^{3}
p_{n}^{(i,j)}\in\mathbb{Z}$, where $l_{n}$ denotes the least common multiple of $\left\{1,2,\ldots,n\right\}$, and
\begin{equation*}
\left( \omega _{i,j}\right) _{4,3}=
\begin{pmatrix}
1 & 0 & 1 \\ 
2 & 0 & 2 \\ 
1 & 0 & 1 \\ 
2 & 1 & 2
\end{pmatrix}.
\end{equation*}
\end{lemma}

\begin{proof}
Since $R_{1,n}^{(i,j)}(t) =\mathcal{O}\left(
t^{-2-\delta _{i,3}}\right) $ as $t\rightarrow \infty $, we have
\begin{equation}
\sum_{k=0}^{n}a_{k,n}^{(i,j)}=\sum_{k=0}^{n}
\mathop{\mathrm{Res}}_{t=-k-1}{R_{1,n}^{(i,j)}(t)}
=-\mathop{\mathrm{Res}}_{t=\infty
}{R_{1,n}^{(i,j)}(t)} =0.
\label{cond_A1}
\end{equation}

Hence, denoting $r_{n}^{(i,j)}=\sum_{t=0}^{\infty}R_{2,n}^{(i,j)}\left(
t\right)$ we get
\begin{eqnarray*}
r_{n}^{(i,j)} &=&\sum_{t=0}^{\infty}\left(\sum_{k=0}^{n-\delta _{i,3}}
\frac{2b_{k,n-\delta _{i,3}}^{(i,j)}}{\left( t+k+1\right) ^{3}}
-\sum_{k=0}^{n}\frac{a_{k,n}^{(i,j)}}{\left(
t+k+1\right) ^{2}}\right) \\
&=&2\sum_{k=0}^{n-\delta _{i,3}}\sum_{l=k+1}^{\infty}\frac{b_{k,n-\delta
_{i,3}}^{(i,j)}}{l^{3}}-\sum_{k=0}^{n}\,\sum_{l=k+1}^{\infty}\frac{
a_{k,n}^{(i,j)}}{l^{2}} \\
&=&2\sum_{k=0}^{n-\delta _{i,3}}b_{k,n-\delta _{i,3}}^{\left( i,j\right)
}\left( \sum_{l=1}^{\infty}\frac{1}{l^{3}}-\sum_{l=1}^{k}\frac{1}{l^{3}}\right)
-\sum_{k=0}^{n}a_{k,n}^{(i,j)}\left( \sum_{l=1}^{\infty}\frac{1}{l^{2}}-\sum_{l=1}^{k}\frac{1}{l^{2}}\right) \\
&=&2\sum_{k=0}^{n-\delta _{i,3}}b_{k,n-\delta _{i,3}}^{\left( i,j\right)
}\sum_{l=1}^{\infty}\frac{1}{l^{3}}-2\sum_{k=1}^{n-\delta _{i,3}}b_{k,n-\delta
_{i,3}}^{(i,j)}\sum_{l=1}^{k}\frac{1}{l^{3}}
+\sum_{k=1}^{n}a_{k,n}^{(i,j)}\sum_{l=1}^{k}\frac{1}{l^{2}},
\end{eqnarray*}
which coincides with \refe{Remainderi,j} by considering the expressions given in \refe{RApproxs}. Finally, using the inclusions \refe{inclu} and the fact that
$l_{n}^{s}\sum_{l=1}^{k}\frac{1}{l^{s}}\in \mathbb{Z}$, $k=0,1,\ldots,n$, $s\in \mathbb{Z}^{+}$, one gets $n^{\omega _{i,j}}q_{n}^{(i,j)}\in \mathbb{Z}$
and $n^{\omega _{i,j}}l_{n}^{3}p_{n}^{(i,j)}\in \mathbb{Z}$, which completes the proof.\end{proof}

\begin{lemma}
The following relation for \refe{Remainderi,j} is
valid
\begin{equation}
r_{n}^{(i,j)}=\frac{1}{2\pi \ima}\int_{-1/2-\ima\infty
}^{-1/2+\ima\infty }R_{1,n}^{(i,j)}\left( \nu \right) \left( \frac{
\pi }{\sin \pi \nu }\right) ^{2}d\nu=\sum_{t=0}^{\infty}R_{2,n}^{\left(
i,j\right) }(t).  \label{IRemainderi,j}
\end{equation}
\end{lemma}

\begin{proof} For the evaluation of the integral \refe{IRemainderi,j} one expresses it as a limit of contour integrals
along the contour $\Omega _{n,i,j}$ that goes along the imaginary line from $
-1/2+\ima L_{n,i,j}$ to $-1/2-\ima L_{n,i,j}$ and then counterclockwise along a
semicircle centered at $-1/2$ from $-1/2+\ima L_{n,i,j}$ to $-1/2-\ima L_{n,i,j}$,
where the semicircle radius $L_{n,i,j}>n+2$. We have taken $L_{n,i,j}$ to be
greater than $n+2$, so that $n+1$ singularities of the integrand function
are enclosed within the curve. The rational function $R_{1,n}^{\left(
i,j\right) }(t)=\mathcal{O}\left( L_{n,i,j}^{-2}\right) $ on the arc of $
\Omega _{n,i,j}$, while the function $\left( \sin \pi z\right) ^{-1}$ is
bounded. Now, by residue theorem one can compute \refe{IRemainderi,j}. Indeed,
\begin{equation*}
\mathop{\mathrm{Res}}_{t=n}{\left( R_{1,n}^{(i,j)}(z)\left( \frac{\pi }{
\sin \pi z}\right) ^{2}\right)}=R_{2,n}^{(i,j)}(n),\quad
n=0,1,2,\ldots ,
\end{equation*}
which can be easily checked by considering the following expansions at the
integers
\begin{eqnarray*}
R_{1,n}^{(i,j)}(z) &=&R_{1,n}^{\left( i,j\right)
}(n)+R_{2,n}^{(i,j)}(n)\left( z-n\right) +\mathcal{O}\left(
\left( z-n\right) ^{2}\right) , \\
\left( \frac{\pi }{\sin \pi z}\right) ^{2} &=&\frac{1}{\left( z-n\right) ^{2}
}+\mathcal{O}\left( 1\right) .
\end{eqnarray*}
Therefore,
\begin{equation*}
r_{n}^{(i,j)}=\frac{1}{2\pi \ima}\int_{-1/2-\ima\infty
}^{-1/2+\ima\infty }R_{1,n}^{(i,j)}\left( \nu \right) \left( \frac{
\pi }{\sin \pi \nu }\right) ^{2}d\nu =\sum_{t=0}^{\infty}R_{2,n}^{\left(
i,j\right) }(t),
\end{equation*}
holds.\end{proof}

\begin{theorem}
There holds the following asymptotic formula
\begin{equation*}
r_{n}^{(i,j)}=\frac{-\pi ^{3/2}\eta^{(i,j)}}{n^{3/2-\delta _{i,4}}2^{1/4}}
\left( \sqrt{2}-1\right) ^{4n}\left(
1+o\left( 1\right) \right),\,\, \left(\eta^{(i,j)}\right)_{4,3}=
\left(\begin{array}{rrr}
-1 & -\frac{|\vartheta|}{\vartheta} & 1 \\ 
1 & \frac{|\vartheta|}{\vartheta} & 1 \\ 
1 & \frac{|\vartheta|}{\vartheta} & -1 \\ 
1 & 1 & -1
\end{array}\right).
\end{equation*}
\label{theoremRemainderi,j}
\end{theorem}

\begin{proof} From expression \refe{IRemainderi,j} one writes
\begin{equation}
r_{n}^{(i,j)} =\frac{1}{2\pi \ima}\int_{1/2-\ima\infty }^{1/2+\ima\infty }R_{1,n}\left( \nu
-1\right) \left( \frac{\pi }{\sin \pi \nu }\right) ^{2}\theta ^{\left(
i,j\right) }\left( \nu -1\right) d\nu.
\label{eqrn}
\end{equation}
Now, taking into account the following estimations 
\begin{equation*}
\log \theta ^{(i,j)}\left(tn+t-1\right) \sim 
\begin{cases}
\displaystyle\log \frac{\left( t+1\right) ^{2-\delta _{i,1}-\delta _{i,4}}}{
\left( t-1\right) ^{2-\delta _{i,1}}}-\delta _{i,4}\log \left( n+1\right) , & 
j=1, \\ 
\displaystyle\log \frac{\left( t+1\right) ^{\delta _{i,3}}\left( t+\vartheta
\right) ^{\delta _{i,2}+1}}{\left( t-1\right) ^{2-\delta _{i,1}}}-\delta
_{i,4}\log \left( n+1\right) , & j=2, \\ 
\displaystyle\log \frac{\left( t+1\right) ^{\delta _{i,3}}\left( \upsilon
t-\chi \right) ^{\delta _{i,2}+1}}{\left( t-1\right) ^{2-\delta _{i,1}}}
-\delta _{i,4}\log \left( n+1\right) , & j=3,
\end{cases}
\label{lnPerturs}
\end{equation*}
it is a matter of straightforward computation to see that for $z=(n+1)t$
\begin{multline*}
\log \frac{\left( 1-z\right) _{n}^{2}}{(z) _{n+1}^{2}}\left( 
\frac{\pi }{\sin \pi z}\right) ^{2}\theta ^{(i,j)}\left(
z-1\right)=\log g^{(i,j)}\left(
t\right) +2\left( n+1\right) f(t) \\
-\left( 2+\delta _{i,4}\right) \log \left( n+1\right) +2\log 2\pi +\mathcal{O}
\left( n^{-1}\right) ,
\end{multline*}
where
\begin{align*}
g^{(i,j)}(t) &=
\begin{cases}
\displaystyle\frac{\left( t+1\right) ^{2-\delta _{i,1}-\delta _{i,4}}}{
\left( t-1\right) ^{2-\delta _{i,1}}}g(t) , & j=1, \\ 
\displaystyle\frac{\left( t+1\right) ^{\delta _{i,3}}\left( t+\vartheta
\right) ^{\delta _{i,2}+1}}{\left( t-1\right) ^{2-\delta _{i,1}}}g\left(
t\right) , & j=2, \\ 
\displaystyle\frac{\left( t+1\right) ^{\delta _{i,3}}\left( \upsilon t-\chi
\right) ^{\delta _{i,2}+1}}{\left( t-1\right) ^{2-\delta _{i,1}}}g\left(
t\right) , & j=3,
\end{cases}
\end{align*}
and
\begin{eqnarray*}
g(t) &=&\frac{1+t}{\left( 1-t\right) t^{2}}, \\
f(t) &=&\left( 1-t\right) \log \left( 1-t\right) +2t\log t-\left(
1+t\right) \log \left( 1+t\right) .
\end{eqnarray*}
Thus, expression \refe{eqrn} transforms into
\begin{equation*}
r_{n}^{(i,j)}=\frac{2\pi\ima}{n^{1+\delta _{i,4}}}
\int_{1/\sqrt{2}-\ima\infty }^{1/\sqrt{2}+\ima\infty }g^{(i,j)}\left(
t\right) e^{2\left( n+1\right) f(t) }\left( 1+\mathcal{O}\left(
n^{-1}\right) \right) dt.
\end{equation*}
The point $t=1/\sqrt{2}$ is the unique maximum point for $\text{Re\,}
f(t) $ on the contour of integration. Therefore, by using Laplace's method we obtain
\begin{equation*}
r_{n}^{(i,j)}=-\dfrac{\pi ^{3/2}\eta^{(i,j)}}{2^{1/4} n^{1+\delta_{i,4}}(n+1)^{1/2}}
\left\vert g^{(i,j)}\left(\sqrt{2^{-1}}\right)\right\vert
\left(\sqrt{2}-1\right)^{4n+4}
\left(1+\mathcal{O}\left( n^{-1}\right)
\right),
\end{equation*}
which gives the required estimation.\end{proof}

The above new rational approximants \refe{RApproxs} prove the irrationality of $\zeta(3)$.
\begin{corollary} (Ap\'ery's Theorem)
The real number $\zeta (3)$ is irrational.
\end{corollary}
\begin{proof} Suppose on the contrary that $\zeta \left( 3\right) =p/q$,
where $p\in \mathbb{Z}$, $q\in\mathbb{N}$, then
\begin{equation*}
qn^{\omega _{i,j}}l_{n}^{3}r_{n}^{(i,j)}=n^{\omega
_{i,j}}l_{n}^{3}q_{n}^{(i,j)}p-qn^{\omega
_{i,j}}l_{n}^{3}p_{n}^{(i,j)},
\end{equation*}
is an integer different from zero. Therefore
\begin{equation*}
1\leq qn^{\omega _{i,j}}l_{n}^{3}\left\vert r_{n}^{\left( i,j\right)
}\right\vert =\mathcal{O}\left( l_{n}^{3}\left( \sqrt{2}-1\right)
^{4n}\right),
\end{equation*}
contradicting the above assumption for $\zeta(3)$, since for any $\varepsilon >0$ and any sufficiently large $n$ the estimation $l_{n}<e^{\left( 1+\varepsilon
\right) n}$ yields $e^{3}\left( \sqrt{2}-1\right)
^{4} <1$. \end{proof}

The above Diophantine approximations to the number $\zeta(3)$ lead to the same irrationality measure given by Ap\'ery's approximants \refe{apery_frac} (see Remark \ref{remar-mu} below).

Next we will investigate other properties of the rational approximants obtained here. In particular, we will deduce the recurrence relation that they verify. 
Indeed, one can use the above functions 
\refe{RationalFi,j} and Zeilberger's algorithm of creative telescoping \cite{Zeilber} to obtain that \refe{Remainderi,j} verifies a second order recurrence relation 
(see \cite{zudilin-2} for the use of these techniques with Ap\'ery's sequence of Diophantine approximants). 
Although we will proceed in a different way. We start by showing a relationship with a simultaneous rational approximation problem to derive in a straightforward manner 
the Wronskian determinant for sequences $\left(p_n^{(i,j)}\right)_{n\geq0}$ and $\left(q_n^{(i,j)}\right)_{n\geq0}$ 
and from it the coefficients involved in the recurrence relation satisfied by these sequences.

\section{Simultaneous rational approximation problem\label{rapp}}
Here we will establish a connection of the aforementioned rational approximants with a simultaneous rational approximation problem near infinity.

Define the following polynomials 
\begin{equation}
A_{n}^{(i,j)}(z)=\sum_{k=0}^{n}a_{k,n}^{(i,j)}z^{k},\quad B_{n-\delta _{i,3}}^{(i,j)}(z) =\sum_{k=0}^{n-\delta _{i,3}}b_{k,n-\delta _{i,3}}^{\left( i,j\right)}z^{k},
\label{A_B-poly}
\end{equation}
with coefficients given in \refe{bk,n} and \refe{ak,n}, respectively. Observe that $b_{0,n-\delta _{i,3}}^{(i,j)}=-\theta ^{(i,j)}\left( -1\right)$, and from equation \refe{cond_A1} one has $A_{n}^{\left(
i,j\right) }\left( 1\right) =0$. 

Thus, from equations \refe{R1_simple} and \refe{R_{2,n}} one gets
\begin{equation}
R_{1,n}^{\left(i,j\right) }(t)=\int_{0}^{1}F_{n}^{(i,j)}(x)x^{t}dx,
\quad R_{2,n}^{(i,j)}(t)=\int_{0}^{1}G_{n}^{(i,j)}(x)x^{t}dx,\quad(\text{Re}\,t>-1),
\label{R1-R2}
\end{equation}
where
\begin{equation*}
F_{n}^{(i,j)}\left( x\right) =A_{n}^{(i,j)}\left(
x\right) +B_{n-\delta _{i,3}}^{(i,j)}\left( x\right) \log x
\text{ \quad }\mbox{and}\text{\quad }G_{n}^{(i,j)}\left(
x\right) =F_{n}^{(i,j)}\left( x\right) \log x.
\end{equation*}
Moreover, the expressions \refe{RationalFi,j} and \refe{R_{2,n}} represent the analytic continuation for the functions given in \refe{R1-R2}.

Considering the zeros of the rational functions \refe{R1_simple} and \refe{R_{2,n}}, which do not dependent on the parameters ($\rho$, $\vartheta$, $\upsilon$, $\chi$, and $\psi$) involved in \refe{Perturbs}, one has the orthogonality conditions
\begin{equation}\begin{array}{c}
\dst\int_{0}^{1}F_{n}^{(i,j)}\left( x\right) x^{k}dx=0,\quad
k=0,\ldots ,n-2+\delta _{i,1},\\
\\
\dst\int_{0}^{1}G_{n}^{(i,j)}\left( x\right) x^{k}dx=0,\quad
k=0,\ldots ,n-2.
\end{array}
\label{eq-orto-12}
\end{equation}
If the coefficients of polynomials \refe{A_B-poly}, i.e. $a_{k,n}^{(i,j)}$ $(k=0,\dots,n)$ and $b_{k,n}^{(i,j)}$ $(k=0,\dots,n-\delta_{i,3})$ were unknown, then the zeros that depend on the parameters $\rho$, $\vartheta$, $\upsilon$, $\chi$, or $\psi$ provide an extra condition to the above underdetermined linear system of equations for these coefficients.

As a consequence of the above orthogonality conditions \refe{eq-orto-12}
\begin{equation}
\begin{array}{c}
\displaystyle\int_{0}^{1}p\left( x\right) \frac{F_{n}^{(i,j)}\left( x\right) }{z-x}dx=p(z) \displaystyle\int_{0}^{1}\frac{
F_{n}^{(i,j)}\left( x\right) }{z-x}dx, \\ 
\\ 
\displaystyle\int_{0}^{1}q\left( x\right) \frac{G_{n}^{\left( i,j\right)
}\left( x\right) }{z-x}dx=q(z) \displaystyle\int_{0}^{1}\frac{
G_{n}^{(i,j)}\left( x\right) }{z-x}dx,
\end{array}
\label{Cons}
\end{equation}
where $p(z)$ and $q(z)$ are arbitrary polynomials of degree at most $n+\delta_{i,1}-1$ and $n-1$, respectively.

Denoting 
\begin{equation*}
r_{n,1}^{(i,j)}(z)=\displaystyle\int_{0}^{1}\frac{
F_{n}^{(i,j)}\left( x\right) }{z-x}dx,\quad
r_{n,2}^{(i,j)}(z)=\displaystyle\int_{0}^{1}\frac{
G_{n}^{(i,j)}\left( x\right) }{z-x}dx,
\end{equation*}
equations \refe{Cons} imply that $r_{n,1}^{(i,j)}(z)=\mathcal{O}\left(
z^{-n-\delta _{i,1}}\right)$ and $r_{n,2}^{(i,j)}(z)=\mathcal{O}\left(
z^{-n}\right)$.

Setting $p(z)=q(z)=1$ in equations \refe{Cons},
by means of a suitable addition and substraction in the numerator of their right hand sides, one obtains
\begin{align}
\int_{0}^{1}\dfrac{F_{n}^{(i,j)}(x)}{z-x}dx&=\int_{0}^{1}\frac{A_{n}^{(i,j)}\left( x\right) +B_{n-\delta _{i,3}}^{(i,j)}\left(
x\right) \log x}{z-x}dx\label{eq1ij}
\\
&=A_{n}^{(i,j)}(z) f_{1}(z)
+B_{n-\delta _{i,3}}^{(i,j)}(z) f_{2}\left(
z\right) -C_{n}^{(i,j)}(z),\notag\\
\int_{0}^{1}\dfrac{G_{n}^{(i,j)}(x)}{z-x}dx&=\int_{0}^{1}\frac{A_{n}^{(i,j)}\left( x\right) +B_{n-\delta _{i,3}}^{(i,j)}\left(
x\right) \log x}{z-x}\log xdx\label{eq2ij}\\
&=A_{n}^{(i,j)}(z) f_{2}(z)
+2B_{n-\delta _{i,3}}^{(i,j)}(z) f_{3}\left(
z\right) -D_{n}^{(i,j)}(z),\notag
\end{align}
where
\begin{align}
f_{k}(z)&=\dfrac{1}{(k-1)!}\int_{0}^{1}\frac{\log ^{k-1}x}{z-x}
dx,\quad k=1,2,3,
\label{f_k}\\
C_n^{(i,j)}(z)&=\int^1_0 \dfrac{A_n^{(i,j)}(z)+B_{n-\delta _{i,3}}^{(i,j)}(z)\log x-F_{n}^{(i,j)}(x)}{z-x}\, dx,
\label{eq-C_n}\\
D_n^{(i,j)}(z)&=\int^1_0 \dfrac{\left(A_n^{(i,j)}(z)+B_{n-\delta_{i,3}}^{(i,j)}(z)\log x\right)\log x-G_{n}^{(i,j)}(x)}{z-x}\, dx.\label{eq-D_n}
\end{align}
Accordingly, for the above system of functions \refe{f_k} and polynomials \refe{A_B-poly}, \refe{eq-C_n}, and \refe{eq-D_n} we have a simultaneous rational approximation problem near infinity. Notice that the solution of this problem depends only on the coefficients of the polynomials $A_{n}^{(i,j)}(z)$ and $B_{n-\delta_{i,3}}^{(i,j)}(z)$, since the coefficients for $z^{-j}$ $(1\leq j\leq n-1+\delta_{i,1})$ in the Laurent series expansion of $A_n^{(i,j)}(z)f_1(z)+B_{n-\delta _{i,3}}^{(i,j)}(z)f_2(z)$ and for $z^{-j}$ $(1\leq j\leq n-1)$ in the series expansion of $A_n^{(i,j)}(z)f_2(z)+2B_{n-\delta _{i,3}}^{(i,j)}(z)f_3(z)$ vanish, while the coefficients for $z^j$ $(0\leq j\leq n)$ coincide with the corresponding coefficients of $C_n^{(i,j)}(z)$ and $D_n^{(i,j)}(z)$, respectively.

\section{Recurrence relation\label{aperyRR}}

In what follows, without loss of generality, we restrict ourself to the particular case $i=1$ and $j=2$. 
The only reason for this restriction is that in general the expressions of coefficients \refe{Coeffs_Rrelation} take significantly more area for displaying them. 
All cases can be dealt with the procedure described in the sequel.

Setting $z=1$ in the simultaneous rational approximation problem \refe{eq1ij}-\refe{eq2ij} and \refe{f_k}-\refe{eq-D_n} yield 
$$
r_{n}^{(1,2)}(1)=2B_{n}^{(1,2)}(1) \zeta(3)-D_{n}^{(1,2)}(1),\quad n\geq1,
$$ where
\begin{equation*}\begin{array}{c}
 \left(2B_{n}^{(1,2)}(1)\right)_{n\geq1}=\left(q_{n}^{(1,2)}\right)_{n\geq1},\quad
\left(D_{n}^{(1,2)}(1)\right)_{n\geq1}=\left(p_{n}^{(1,2)}\right)_{n\geq 1},\\
\left(r_{n,2}^{(1,2)}(1)\right)_{n\geq 1}=\left(r_{n}^{(1,2)}\right)_{n\geq 1}.
\end{array}
\end{equation*}
Now, we will obtain an explicit expression for the Wronskian determinant involving the above sequences, i.e.
\begin{equation*}
W\left( q_{n}^{(1,2)},r_{n}^{(1,2)}\right) =\det 
\begin{pmatrix}
q_{n}^{(1,2)} & r_{n}^{(1,2)} \\ 
q_{n+1}^{(1,2)} & r_{n+1}^{(1,2)}
\end{pmatrix}
=-W\left( q_{n}^{(1,2)},p_{n}^{(1,2)}\right).
\end{equation*}

\begin{lemma}\label{wrons} The following relation
\begin{equation}
W\left( q_{n},p_{n}\right) =\frac{2\mathcal{N}_{n}}{n^{3}\left( n+1\right) ^{3}}\neq 0,\quad n\geq 1,
\label{Wronsk}
\end{equation}
where
\begin{multline*}
\mathcal{N}_{n}=(24\vartheta ^{2}n^{3}+30\vartheta ^{2}n^{2}+16\vartheta
^{2}n+3\vartheta ^{2}+9\vartheta n^{2}+5\vartheta n \\
+\vartheta -12n^{3}-21n^{2}-11n-2),\quad \vartheta \in \mathbb{N}\setminus\{1\},
\end{multline*}
holds.
\end{lemma}
Here for proving \refe{Wronsk} we will follow the procedure indicated in \cite[Lemma 3.2, pp. 7-8]{arvesu}
\begin{proof} Let us consider the integral
\begin{equation*}
I_{n}=2\int_{0}^{1}\frac{F_{n}^{(1,2)}\left( x\right) F_{n+1}^{(1,2)}\left( x\right) }{1-x}
dx,\quad n\geq1.
\end{equation*}
Using formulas \refe{Cons} and the fact that $A_{n}^{(1,2)}(1)=0$, one gets
\begin{equation}
I_{n} =2B_{n}^{(1,2)}\left( 1\right) \int_{0}^{1}\frac{G_{n+1}^{(1,2)}\left( x\right) }{1-x}dx =q_{n}^{(1,2)}r_{n+1}^{(1,2)}.\label{In1}
\end{equation}
On the other hand,
\begin{align*}
 \dfrac{F_{n}^{(1,2)}(x)F_{n+1}^{(1,2)}(x)}{1-x}&=B_{n+1}^{(1,2)}(1)\dfrac{G_{n}^{(1,2)}(x)}{1-x}
-F_{n}^{(1,2)}(x)\left(\tilde{A}_{n}^{(1,2)}(x)+\tilde{B}_{n}^{(1,2)}(x)\log x\right),
\end{align*}
where 
\begin{equation*}
\tilde{A}^{(1,2)}_{n}(x)=\frac{A_{n+1}^{(1,2)}(1)-A_{n+1}^{(1,2)}(x)}{1-x},\quad
\tilde{B}_{n}^{(1,2)}(x)
=\frac{B_{n+1}^{(1,2)}(1)-B_{n+1}^{(1,2)}(x)}{1-x}.
\end{equation*}
Hence,
\begin{eqnarray}
I_{n} &=&2\int_{0}^{1}\frac{F_{n}^{(1,2)}\left( x\right) F_{n+1}^{(1,2)}\left( x\right) }{1-x
}dx =q_{n+1}^{(1,2)}r_{n}^{(1,2)} -2\left[a_{n+1,n+1}^{(1,2)}R_{1,n}^{(1,2)}\left(
n\right)\right.  \label{In2}\\
&+&\left.b_{n+1,n+1}^{(1,2)}\left(R_{2,n}^{(1,2)}\left(
n\right) +R_{2,n}^{(1,2)}(n-1)\right)+b_{n,n+1}^{\left(
1,2\right) }R_{2,n}^{(1,2)}\left( n-1\right)\right].\notag
\end{eqnarray}
Equating \refe{In1} and \refe{In2} one obtains
\begin{eqnarray}
W\left( q_{n},r_{n}\right)  &=&-2\left[a_{n+1,n+1}^{(1,2)}R_{1,n}^{(1,2)}\left(
n\right)\right.  \notag \\
&+&\left.b_{n+1,n+1}^{(1,2)}\left(R_{2,n}^{(1,2)}\left(
n\right) +R_{2,n}^{(1,2)}(n-1)\right)+b_{n,n+1}^{\left(
1,2\right) }R_{2,n}^{(1,2)}\left( n-1\right)\right].\notag \\
&=&-\frac{2\mathcal{N}_{n}}{n^{3}\left( n+1\right) ^{3}},\quad n\geq 1.
\notag
\end{eqnarray}
Finally, observe that $\mathcal{N}_{n}$ is a polynomial of degree $2$ in $\vartheta$, with non-integer zeros. Indeed, when $n\to\infty$ their zeros tend to $1/\sqrt{2}$ and $-1/\sqrt{2}$, respectively. Thus, $\mathcal{N}_{n}\not=0$ for $\vartheta\in\mathbb{N}\setminus\{1\}$, and $n\geq1$,
which completes the proof.\end{proof}

\begin{theorem}
The sequences $\left( p_{n}^{(1,2)}\right) _{n\geq 1}$, $\left(
q_{n}^{(1,2)}\right) _{n\geq 1}$ and $\left( r_{n}^{\left(
1,2\right) }\right) _{n\geq 1}$ verify the following second order
recurrence relation
\begin{equation}
\alpha _{n}y_{n+2}+\beta _{n}y_{n+1}+\gamma _{n}y_{n}=0,\quad n=1,2,\ldots,\quad \vartheta \in \mathbb{N}\setminus\{1\},
\label{R_relation1,2}
\end{equation}
where
\begin{equation}
\begin{array}{c}
\alpha _{n}=(n+2)^{3}(24\vartheta ^{2}n^{3}+30\vartheta
^{2}n^{2}+16\vartheta ^{2}n+3\vartheta ^{2}+9\vartheta n^{2}+5\vartheta
n+\vartheta \\ 
-12n^{3}-21n^{2}-11n-2), \\ 
\\ 
\beta _{n}=-2(408\vartheta ^{2}n^{6}+2346\vartheta ^{2}n^{5}+5336\vartheta
^{2}n^{4}+6130\vartheta ^{2}n^{3}+3810\vartheta ^{2}n^{2}\\
+1268\vartheta ^{2}n +172\vartheta ^{2}+153\vartheta n^{5}+769\vartheta n^{4}+1417\vartheta
n^{3}+1143\vartheta n^{2}+382\vartheta n\\
+52\vartheta -204n^{6}-1275n^{5}-3181n^{4}-4011n^{3}-2667n^{2}-886n-120), \\ 
\\ 
\gamma _{n}=n^{3}(24\vartheta ^{2}n^{3}+102\vartheta ^{2}n^{2}+148\vartheta
^{2}n+73\vartheta ^{2}+9\vartheta n^{2}+23\vartheta n+15\vartheta \\ 
-12n^{3}-57n^{2}-89n-46).
\end{array}
\label{Coeffs_Rrelation}
\end{equation}
\label{TRR}
\end{theorem}

\begin{proof} From previous Lemma \ref{wrons} we write
\begin{eqnarray*}
q_{n+1}^{(1,2)}r_{n+2}^{(1,2)} &=&q_{n+2}^{\left(
1,2\right) }r_{n+1}^{(1,2)}-W_{n+1}\Leftrightarrow
r_{n+2}^{(1,2)}=\frac{q_{n+2}^{(1,2)}}{
q_{n+1}^{(1,2)}}r_{n+1}^{(1,2)}-\frac{W_{n+1}}{
q_{n+1}^{(1,2)}}, \\
q_{n}^{(1,2)}r_{n+1}^{(1,2)} &=&q_{n+1}^{\left(
1,2\right) }r_{n}^{(1,2)}-W_{n}\Leftrightarrow \frac{
q_{n}^{(1,2)}}{q_{n+1}^{(1,2)}}r_{n+1}^{\left(
1,2\right) }=r_{n}^{(1,2)}-\frac{W_{n}}{q_{n+1}^{\left(
1,2\right) }},
\end{eqnarray*}
where
\begin{equation*}
W_{n}=\frac{2\mathcal{N}_{n}}{n^{3}\left( n+1\right) ^{3}}.
\end{equation*}
Thus, multiplying the first equation by $W_{n}$, the second one by $
-W_{n+1}$, and adding both equations one gets
\begin{equation*}
W_{n}r_{n+2}^{(1,2)}-\left( W_{n}\frac{q_{n+2}^{\left(
1,2\right) }}{q_{n+1}^{(1,2)}}+W_{n+1}\frac{q_{n}^{\left(
1,2\right) }}{q_{n+1}^{(1,2)}}\right) r_{n+1}^{\left(
1,2\right) }+W_{n+1}r_{n}^{(1,2)}=0,\quad n\geq 1.
\end{equation*}
Then, multiplying this equation by $n^{3}\left( n+1\right)
^{3}\left( n+2\right) ^{3}$, one has
\begin{equation*}
\alpha _{n}r_{n+2}^{(1,2)}-\tilde{\beta}_{n}r_{n+1}^{\left(
1,2\right) }+\gamma _{n}r_{n}^{(1,2)}=0,\quad n=1,2,\dots ,
\end{equation*}
where $\alpha _{n}$ and $\gamma _{n}$ are given in \refe{Coeffs_Rrelation}, and
\begin{equation}
\tilde{\beta}_{n}=\alpha _{n}\frac{q_{n+2}^{(1,2)}}{
q_{n+1}^{(1,2)}}+\gamma _{n}\frac{q_{n}^{(1,2)}}{
q_{n+1}^{(1,2)}}.  \label{beta_n_B_n}
\end{equation}
This implies the verification of the second order recurrence relation
\begin{equation}
\alpha _{n}y_{n+2}-\tilde{\beta}_{n}y_{n+1}+\gamma _{n}y_{n}=0,\quad
n=1,2,\dots,\label{eq-ttrr-generic}
\end{equation}
by the sequences $(q_{n}^{(1,2)})_{n\geq 1}$, $(p_{n}^{\left(
1,2\right) })_{n\geq 1}$ and $(r_{n}^{(1,2)})_{n\geq 1}$. Furthermore,
from equation \refe{eq-ttrr-generic} one can write
\begin{align*}
\tilde{\beta}_{n}=\alpha _{n}\frac{p_{n+2}^{(1,2)}}{
p_{n+1}^{(1,2)}}+\gamma _{n}\frac{p_{n}^{(1,2)}}{
p_{n+1}^{(1,2)}}& =\alpha _{n}\frac{r_{n+2}^{(1,2)}
}{r_{n+1}^{(1,2)}}+\gamma _{n}\frac{r_{n}^{(1,2)}}{
r_{n+1}^{(1,2)}}  =\alpha _{n}\frac{q_{n+2}^{(1,2)}}{q_{n+1}^{\left( 1,2\right)
}}+\gamma _{n}\frac{q_{n}^{(1,2)}}{q_{n+1}^{(1,2)}}.
\end{align*}
Notice that from \refe{beta_n_B_n} the sequence $\tilde{\beta}_{n}/n^{6}$
converges when $n\rightarrow \infty $ (see \cite[Theorem 3.3, p.11]{arvesu}). Thus, by setting $n=1,2,\ldots ,7$ in relation \refe{beta_n_B_n} one gets a linear system of equations for determining explicitly the coefficients of $\tilde{\beta}_{n}=an^{6}+bn^{5}+cn^{4}+dn^{3}+en^{2}+fn+g$.  The solution of the resulting system gives us the coefficient $\beta _{n}$ given in \refe{Coeffs_Rrelation}, where $\tilde{\beta}_{n}=-\beta_{n}$. The theorem is completely proved.\end{proof}

\begin{remark}\label{remar-mu} 
The characteristic equation for the recurrence relation \refe{R_relation1,2} is $t^2-34t+1$,  which coincides with the one derived from the Ap\'ery's recurrence relation \refe{apery-diff-eq-1}. Accordingly, our Diophantine approximations do not improve the irrationality measure $\mu$ obtained in \cite{apery,beuker-1,nesterenko}, i.e  $\mu=1+(\log(\sqrt{2}+1)^4+3)/(\log(\sqrt{2}+1)^4-3)$.
\end{remark}

An important consequence of the above theorem is the continued fraction representation of the number $\zeta(3)$. Below we present one of 
various possible continued fraction representations that can be deduced from our approach (see the rational functions \refe{RationalFi,j} as 
well as the simultaneous rational approximation problem \refe{eq1ij}-\refe{eq2ij}).

Recall that two irregular continued fractions
\begin{equation*}
a_{0}+\frac{b_{1}\mid }{\mid a_{1}}+\frac{b_{2}\mid }{\mid a_{2}}+\frac{
b_{3}\mid }{\mid a_{3}}+\cdots +\frac{b_{n}\mid }{\mid a_{n}},\quad
a_{0}^{\prime }+\frac{b_{1}^{\prime }\mid }{\mid a_{1}^{\prime }}+\frac{
b_{2}^{\prime }\mid }{\mid a_{2}^{\prime }}+\frac{b_{3}^{\prime }\mid }{\mid
a_{3}^{\prime }}+\cdots +\frac{b_{n}^{\prime }\mid }{\mid a_{n}^{\prime }},
\end{equation*}
are said to be equivalent if there exists a non-zero sequence $\left(
c_{n}\right) _{n\geq 0}$, with $c_{0}=1$, such that (see \cite[p. 20]{Jones})
\begin{equation*}
a_{n}^{\prime }=c_{n}a_{n},\quad n=0,1,2,\ldots ,\quad b_{n}^{\prime
}=c_{n}c_{n-1}b_{n},\quad n=1,2,\ldots 
\end{equation*}
Furthermore, if $\left( p_{n}\right) _{n\geq -1}$ and $\left( q_{n}\right) _{n\geq -1}$
are two sequences such that $q_{-1}=0$, $p_{-1}=q_{0}=1$ and $
p_{n}q_{n-1}-p_{n-1}q_{n}\neq 0$ for $n=0,1,2,\ldots $, then there exists a
unique irregular continued fraction
\begin{equation}
a_{0}+\frac{b_{1}\mid }{\mid a_{1}}+\frac{b_{2}\mid }{\mid a_{2}}+\frac{
b_{3}\mid }{\mid a_{3}}+\cdots +\frac{b_{n}\mid }{\mid a_{n}},  \label{ICF}
\end{equation}
whose $n$-th numerator is $p_{n}$ and $n$-th denominator is $q_{n}$, for
each $n\geq 0$. More precisely (see \cite[p. 31]{Jones})
\begin{equation*}
a_{0}=p_{0},\quad a_{1}=q_{1},\quad b_{1}=p_{1}-p_{0}q_{1},
\end{equation*}
\begin{equation*}
a_{n}=\frac{p_{n}q_{n-2}-p_{n-2}q_{n}}{p_{n-1}q_{n-2}-p_{n-2}q_{n-1}},\quad
b_{n}=\frac{p_{n-1}q_{n}-p_{n}q_{n-1}}{p_{n-1}q_{n-2}-p_{n-2}q_{n-1}},\quad
n=2,3,\ldots
\end{equation*}

\begin{corollary}
There holds the following continued fraction expansion
\begin{equation}
\zeta (3)=\frac{9\mid }{\mid 8}+\frac{-184\mid }{\mid 359}+\frac{-30672\mid 
}{\mid \quad \mathcal{Q}_{3}\quad }+\frac{\mathcal{P}_{4}\mid }{\mid 
\mathcal{Q}_{4}}+\cdots +\frac{\mathcal{P}_{n}\mid }{\mid \mathcal{Q}_{n}}
+\cdots ,  \label{ICF2}
\end{equation}
where
\begin{equation*}
\mathcal{P}_{n}=-9(n-2)^{3}(n-1)^{3}\left( 28n^{3}-213n^{2}+543n-464\right)
(28n^{3}-45n^{2}+27n-6),
\end{equation*}
and
\begin{equation*}
\mathcal{Q}
_{n}=6(476n^{6}-2907n^{5}+7077n^{4}-8715n^{3}+5715n^{2}-1926n+264).
\end{equation*}
\end{corollary}

\begin{proof} Setting $p_{-1}=q_{0}=1$, $p_{0}=q_{-1}=0$, and 
$\vartheta =2$ one gets
\begin{equation}
a_{1}=8,\quad a_{2}=359/24, \quad b_{1}=9,\quad b_{2}=-23/3.
\label{InitialE}
\end{equation}
Moreover, from the recurrence relation \refe{R_relation1,2} one has
\begin{equation*}
y_{n}=-\frac{\beta _{n-2}}{\alpha _{n-2}}y_{n-1}-\frac{\gamma _{n-2}}{\alpha
_{n-2}}y_{n-2}.
\end{equation*}
Thus, we have constructed the elements of the irregular continued fraction \refe{ICF} which satisfy \refe{InitialE}, i.e.
\begin{equation*}
a_{n}=-\frac{\beta _{n-2}}{\alpha _{n-2}},\quad b_{n}=-\frac{\gamma _{n-2}}{
\alpha _{n-2}},\quad n\geq 3.
\end{equation*}
With the
choice $c_{0}=c_{1}=1$, $c_{2}=24$ and $c_{n}=-\alpha _{n-2}$, for $n\geq 3$, we obtain the irregular continued fraction \refe{ICF2}.\end{proof}

\section{Comparison of results}
Denote $\pi^{(i,j)}_{n}=p^{(i,j)}_{n}/q^{(i,j)}_{n}$, where
the integers $i$, $j$ are such that $0\leq i\leq4$, $0\leq j\leq3$. By 
$\pi^{(0,0)}_{n}$ we denote the Ap\'ery's approximants, where
\begin{equation*}
\begin{array}{c}
\dst q^{(0,0)}_{n}=\sum_{k=0}^{n}b^{(n)}_{k},\quad b^{(n)}_{k}=\binom{n +k}{k}^2\binom{n}{k}^2,\\
p^{(0,0)}_{n}\dst=\sum_{k=1}^{n}\left(b^{(n)}_{k}H_{k}^{(3)}-a^{(n)}_{k}H^{(2)}_{k}\right),\quad a^{(n)}_{k}=\left(H_{n + k} - 2H_{k} + H_{n-k}\right)b^{(n)}_k.
\end{array}
\end{equation*}

In Figure \ref{graph1}, a comparison between twelve Diophantine approximants $\pi^{(i,j)}_{n}$ corresponding to  different choices of parameters and Ap\'ery's approximants $\pi^{(0,0)}_{n}$ is given. We illustrate this comparison by means of a rectangular array of squares formed by thirteen rows and nine columns. We use a grayscale output, in which the color of each square is determined by the value of the function
\begin{equation}
f(\pi^{(i,j)}_{n})=\abs{\left(\log\abs{\zeta(3)-\pi^{(i,j)}_{n}}\right)^{-1}},
\quad 0\leq i\leq4,\,\,0\leq j\leq3,\quad n=2,\dots10,
\label{fx}
\end{equation}
ranges from $0.0144346\dots$ to $0.137009\dots$. Moreover, the values close to the minimum of \refe{fx} are shown as white squares while its maxima are shown as black squares. Indeed, ten iterations (see columns in Figure \ref{graph1}) are enough to reveal the high accuracy of our results.  Clearly, in Figure \ref{graph1} the darkness decreases as the number of iterations grows, which is in accordance with the analytical results given in Sections \ref{section2} and \ref{aperyRR}. See also Tables~\ref{Apery-P11}--\ref{Apery-P13} below, in which a comparison of rates of convergence of four selected cases among the above twelve rational approximants, namely $\pi^{(0,0)}_{n}$, $\pi^{(1,1)}_{n}$ for $\rho =2$, $\pi^{(1,2)}_{n}$ for $\vartheta=2$, and $\pi^{(1,3)}_{n}$ for $\upsilon =\chi =\psi =1$ is presented.

\section{Concluding remarks\label{conclu}}

We have obtained a set of rational approximants leading to the irrationality of $\zeta(3)$, with nice rates of convergence to $\zeta(3)$ --as fast as Ap\'ery's approximants. The starting point in our approach is a family of rational functions \refe{RationalFi,j} (which can be considered as modifications of the rational function introduced firstly by Beukers in \cite[p. 97]{beuker-3} and later by Nesterenko in \cite{nesterenko,nesterenko-2}). We have linked this approach with a simultaneous rational approximation problem \refe{eq1ij}-\refe{eq2ij}, which allows to compute in a straightforward manner the coefficients of the second order recurrence relation for the sequences of numerator and denominator of Diophantine approximants to $\zeta(3)$ based on the Wronskian determinant for these sequences. However, small variations in aforementioned starting point might lead to rational approximants that does not prove the irrationality of $\zeta(3)$, as illustrate below. Let
\begin{equation}
R_{1,n}^{\left( 1\right) }=\frac{\left( -t\right) _{n}^{2}}{\left(
t+1\right) _{n+1}^{2}}\left( \frac{t+n+1}{t+n+2}\right),
\label{tildeR1}
\end{equation}
and
\begin{equation}
R_{1,n}^{\left( 2\right) }=\frac{\left( -t\right) _{n}^{2}}{\left(
t+1\right) _{n+1}^{2}}\left( \frac{a_{n}t^{2}+b_{n}}{t-n+1}\right) ,
\label{tildeR2}
\end{equation}
where $a_{n}=4n\left(2H_{n}-H_{2n-1}\right)-1$, $b_{n}=(n+1) a_{n}-2n$. These rational functions (based on the approach given in Sections \ref{section2} and \ref{aperyRR}) lead in general to sequences of rationals $\left(p_n^{(i)}\right)_{n\geq0}$ and $\left(q_n^{(i)}\right)_{n\geq0}$, 
$i=1,2$, which verify a second order recurrence relation with the same characteristic polynomial given by Ap\'ery's equation \refe{apery-diff-eq-1}, nonetheless they do not prove the irrationality of $\zeta(3)$. Although the convergence of $\left(p_n^{(i)}/q_n^{(i)}\right)_{n\geq0}$ to the number $\zeta(3)$ is good enough as is depicted in Figure \ref{fig_tildes_R1_R_2}. Their rates of convergence to $\zeta(3)$ are as good as Ap\'ery's approximants.

The principle that makes the modified functions \refe{RationalFi,j}, or equivalently the simultaneous rational approximation problems \refe{eq1ij}-\refe{eq2ij}, effective for proving the irrationality of $\zeta(3)$ is not clear yet. Whit this paper we begin to understand this question. It would be interesting to obtain other rational approximants with different irrationality measures in the context provided here by a constructive simultaneous rational approximation problem. More investigations need to be done in this direction. A well understanding of these phenomena might help in the study of other modifications of rational functions involved in the arithmetic properties of the numbers $\zeta(2k+1)$, $k=2,3,\dots$

Finally, our interest for constructing infinitely many rational approximants to $\zeta(3)$ proving its irrationality is motivated by a more deeper question, namely the transcendence of $\zeta(3)$, but we are yet uncertain about this question.

\begin{figure}[htb!]
\centering%
\includegraphics[scale=.9]{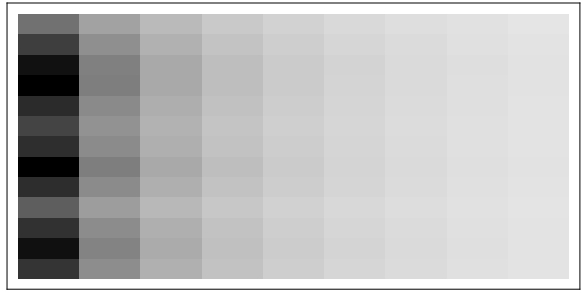}
\caption{From left to right are displayed in grayscale output the values of function \refe{fx} for $n=2,\dots,10$, and from top to bottom its arguments are:
$\pi^{(0,0)}_{n}$, $\pi^{(1,1)}_{n}$ for $\rho=2$, 
$\pi^{(2,1)}_{n}$ for $\rho=4$, 
$\pi^{(3,1)}_{n}$ for $\rho=913$, 
$\pi^{(4,1)}_{n}$ for $\rho=23$, 
$\pi^{(1,2)}_{n}$ for $\vartheta=2$,
$\pi^{(2,2)}_{n}$ for $\vartheta=784$, 
$\pi^{(3,2)}_{n}$ for $\vartheta=93$, 
$\pi^{(4,2)}_{n}$ for $\vartheta=57$, 
$\pi^{(1,3)}_{n}$ for $\upsilon=\chi=\psi=1$, 
$\pi^{(2,3)}_{n}$ for $\upsilon=49$, $\chi=891$, $\psi=97$, 
$\pi^{(3,3)}_{n}$ for $\upsilon=413$, $\chi=732$, $\psi=231$,
$\pi^{(4,3)}_{n}$ for $\upsilon=713$, $\chi=3427$, $\psi=231$. For $n=10$ (last column) is observed a high level of 
coincidence between all Diophantine approximants.}
\label{graph1}
\end{figure}

\begin{table}[ht]\caption{Comparison between Ap\'{e}ry's rational approximants $\pi^{(0,0)}_{n}$ and rational approximants $\pi^{(1,1)}_{n}$ for $\rho =2$.}\label{Apery-P11}
\renewcommand\arraystretch{1.5}
\noindent
\[
\begin{array}{|c|c|c|c|c|}
\hline
n & \pi^{(0,0)}_{n} & \zeta \left( 3\right) -\pi^{(0,0)}_{n} & \pi^{(1,1)}_{n} & \zeta \left( 3\right)-\pi^{(1,1)}_{n} \\ \hline
2 & \frac{351}{292} & 2.109\times 10^{-6} & \frac{1327}{1104} & 
0.00006 \\ \hline
3 & \frac{62531}{52020} & 1.968\times 10^{-9} & \frac{104377}{86832}
& 5.776\times 10^{-8} \\ \hline
4 & \frac{11424695}{9504288} & 1.778\times 10^{-12} & \frac{58624219}{48769920} & 5.211\times 10^{-11} \\ \hline
\vdots & \vdots & \vdots & \vdots & \vdots \\ \hline
50 & \cdot  & 2.795\times 10^{-153} & \cdot  & 9.250\times 10^{-152}
\\ \hline
\end{array}
\]
\end{table}

\begin{table}[ht]
\caption{Comparison between Ap\'{e}ry's rational approximants $\pi^{(0,0)}_{n}$ and rational approximants $\pi^{(1,2)}_{n}$ for $\vartheta =2$.}
\label{Apery-P12}
\renewcommand\arraystretch{1.5}
\noindent\[
\begin{array}{|c|c|c|c|c|}
\hline
n & \pi^{(0,0)}_{n} & \zeta \left( 3\right) -\pi^{(0,0)}_{n} & \pi^{(1,2)}_{n} & \zeta \left( 3\right)-\pi^{(1,2)}_{n} \\ \hline
2 & \frac{351}{292} & 2.109\times 10^{-6} & \frac{1077}{896} & 
0.00004 \\ \hline
3 & \frac{62531}{52020} & 1.968\times 10^{-9} & \frac{1987}{1653} & 
3.686\times 10^{-8} \\ \hline
4 & \frac{11424695}{9504288} & 1.778\times 10^{-12} & \frac{34774333}{28929024} & 3.006\times 10^{-11} \\ \hline
\vdots & \vdots & \vdots & \vdots & \vdots \\ \hline
50 & \cdot & 2.795\times 10^{-153} & \cdot & 3.505\times 10^{-152} \\ \hline
\end{array}
\]
\end{table}

\begin{table}[ht]
\caption{Comparison between Ap\'{e}ry's rational approximants $\pi^{(0,0)}_{n}$ and rational approximants $\pi^{(1,3)}_{n}$ for $\upsilon =\chi =\psi=1$.}\label{Apery-P13}
\renewcommand\arraystretch{1.5}
\noindent\[
\begin{array}[b]{|c|c|c|c|c|}
\hline
n & \pi^{(0,0)}_{n} & \zeta \left( 3\right) -\pi^{(0,0)}_{n} & \pi^{(1,3)}_{n} & \zeta \left( 3\right)
-\pi^{(1,3)}_{n} \\ \hline
2 & \frac{351}{292} & 2.109\times 10^{-6} & \frac{2231}{1856} & 
9.489\times 10^{-6} \\ \hline
3 & \frac{62531}{52020} & 1.968\times 10^{-9} & \frac{783217}{651564}
& 6.216\times 10^{-9} \\ \hline
4 & \frac{11424695}{9504288} & 1.778\times 10^{-12} & \frac{118221931}{98349696} & 4.550\times 10^{-12} \\ \hline
\vdots  & \vdots  & \vdots  & \vdots  & \vdots  \\ \hline
50 & \cdot  & 2.795\times 10^{-153} & \cdot  & 3.114\times 10^{-153}
 \\ \hline
\end{array}
\]
\end{table}

\begin{figure}[htb!]
\centering%
\includegraphics[scale=1.2]{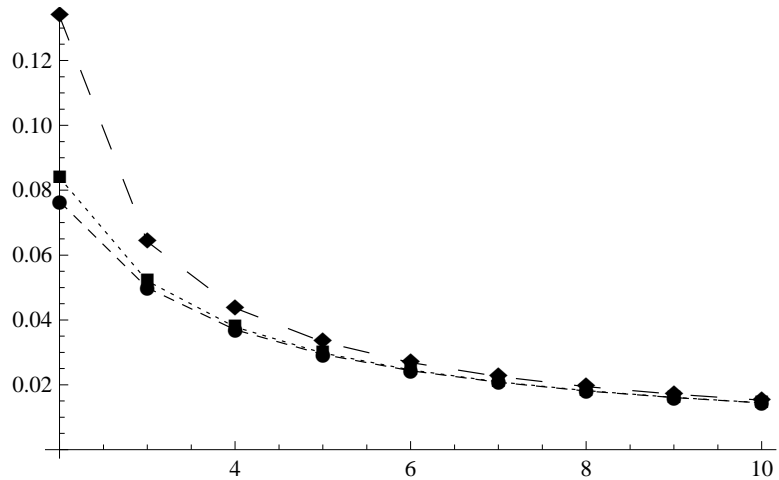}
\caption{Function \refe{fx} is plotted for $n=2,\dots,10$ in the following three situation: 
Symbol $\bullet$ is used for depicting Ap\'ery's approximants, 
while {\tiny$\blacksquare$} and $\blacklozenge$ are used for the approximants 
derived from \refe{tildeR1} and \refe{tildeR2}, respectively.}
\label{fig_tildes_R1_R_2}
\end{figure}

\end{document}